\documentclass[11pt]{article}
\usepackage{psfrag,amsmath}
\usepackage{amssymb}
\usepackage{mathrsfs}
\usepackage{amsfonts}
\usepackage{amsmath}
\usepackage{mathrsfs,color}
\usepackage{epstopdf}
\usepackage{graphicx}
\usepackage{mathrsfs,amscd,amssymb,amsthm,amsmath,bm,url}

\makeatletter

\renewcommand{\@seccntformat}[1]{{\csname the#1\endcsname}{\normalsize .}\hspace{.5em}}
\makeatother

\def \[{\begin{equation}}
\def \]{\end{equation}}

\newtheorem{theorem}{Theorem}[section]

\newtheorem{lemma}[theorem]{Lemma}
\newtheorem{corollary}[theorem]{Corollary}

\begin{document}

\setlength{\baselineskip}{15pt}

\begin{center}{\Large \bf Note on VDB Topological Indices of \boldmath$k$-Cyclic Graphs}
\vspace{4mm}

{\large Hechao Liu$^{a,}$\footnote{Corresponding author.\\
\hspace*{5mm}E-mail: hechaoliu@yeah.net; hechaoliu@m.scnu.edu.cn (H. Liu),\ duzn@m.scnu.edu.cn (Z. Du),\ fayger@qq.com (Y. Huang),\ hlchen@ccsu.edu.cn (H. Chen),\ sureshkako@gmail.com (S. Elumalai)}, \ \ Zenan Du$^a$, \ \ Yufei Huang$^b$,\ \ Hanlin Chen$^c$,\ \ Suresh Elumalai$^{d}$}\vspace{2mm}

$^a$ School of Mathematical Sciences, South China Normal University, Guangzhou, 510631, P. R. China

$^b$ Department of Mathematics Teaching, Guangzhou Civil Aviation College, Guangzhou, 510403, P. R. China

$^c$  School of Mathematics, Changsha University, Changsha, Hunan 410022, P. R. China

$^d$ Department of Mathematics, College of Engineering and Technology, SRM Institute of Science and Technology, Kattankulathur, Chengalpet 603 203, India
\end{center}

\noindent {\bf Abstract}:\ Let $G$ be a connected graph with $n$ vertices and $m$ edges.
The vertex-degree-based topological index (VDB) (or graphical function-index) $TI(G)$ of $G$ with edge-weight function $I(x,y)$ is defined as
$$TI(G)=\sum\limits_{uv\in E(G)}I(d_{u},d_{v}),$$
where $I(x,y)>0$ is a symmetric real function with $x\geq 1$ and $y\geq 1$, $d_{u}$ is the degree of vertex $u$ in $G$.

In this note, we deduce a number of previously established results, and state a few new.
For a VDB topological index $TI$ with the property $P^{*}$, we can obtain the minimum $k$-cyclic (chemical) graphs for $k\geq3$, $n\geq 5(k-1)$. These VDB topological indices include the Sombor index, the general Sombor index, the $p$-Sombor index, the general sum-connectivity index and so on.
Thus this note extends the results of Liu et al. [H. Liu, L. You, Y. Huang, Sombor index of c-cyclic chemical graphs, MATCH Commun. Math. Comput. Chem. 90 (2023) 495-504] and Ali et al. [A. Ali, D. Dimitrov, Z. Du, F. Ishfaq, On the extremal graphs for general sum-connectivity index $(\chi_{\alpha})$ with given cyclomatic number when $\alpha>1$, Discrete Appl. Math. 257 (2019) 19-30].

\vspace{2mm} \noindent{\it Keywords:}
VDB topological index; $k$-cyclic graph; $k$-cyclic chemical graph
\vspace{2mm}


\section{\normalsize Introduction}\setcounter{equation}{0}
In this paper, all notations and terminologies used but not defined can refer to Bondy and Murty \cite{bond2008}.
Let $G$ be a connected graph with vertex set $V(G)$ and edge set $E(G)$, where $|V(G)|=n$ and $|E(G)|=m$.
Let $I(x,y)>0$ be a symmetric real function with $x\geq 1$ and $y\geq 1$, $d_{u}$ the degree of vertex $u$ in $G$. The vertex-degree-based topological index (VDB for short)(or graphical function-index) $TI(G)$ of $G$ with edge-weight function $I(x,y)$ was defined as \cite{guan2013}
\begin{equation}\label{eq:11}
TI(G)=\sum_{uv\in E(G)}I(d_{u},d_{v}).
\end{equation}
When $I(x,y)$ is $\sqrt{x^{2}+y^{2}}$, $(x^{2}+y^{2})^{\alpha}$, $(x^{p}+y^{p})^{\frac{1}{p}}$, $(x+y)^{\alpha}$, we call $TI(G)$ as the Sombor index \cite{gumn2021}, general Sombor index \cite{huzh2022}, $p$-Sombor index \cite{trdo2021}, general sum-connectivity index \cite{zhtr2010}, respectively.

The exponential vertex-degree-based topological index $e^{TI}(G)$ of $G$ with edge-weight function $I(x,y)$ was defined as \cite{raja2019}
\begin{equation}\label{eq:12}
e^{TI}(G)=\sum_{uv\in E(G)}e^{I(d_{u},d_{v})}.
\end{equation}

In \cite{racr2014}, Rada et al. determined the bound of $TI$ over the set of graphs with $n$ vertices.
As an application, they found the extremal values of the general Randi\'{c} index $R_{t}$ when $t\in (-1,-\frac{1}{2})$.
Cruz et al. \cite{crra2019,crrs2022} determined extremal trees and unicyclic graphs for (exponential) vertex-degree-based topological indices.
Cruz et al. \cite{crmr2020} determined extremal values of (exponential) vertex-degree-based topological indices over chemical trees.
By using the majorization theory, Yao et al. \cite{yald2019} presented a uniform
method to some extremal results together with its corresponding extremal graphs
for vertex-degree-based invariants among the class of trees, unicyclic graphs and bicyclic
graphs with fixed number of independence number and/or matching number, respectively.

Recently, Hu et al. \cite{hllp2022} obtained some upper bounds and lower bounds for the topological index $TI(G)$ and give some graphs of given order and size achieving the bounds.
Among bipartite graphs with given order and matching number/vertex cover number/edge cover
number/independence number, among multipartite graphs with given order, and among graphs
with given order and chromatic number, Vetr\'{i}k \cite{vetr2023} presented the graphs having the maximum degree-based-index if that index satisfies certain conditions. They also show that those conditions are satisfied by the general sum-connectivity index $\chi_{\alpha}$ for all or some $\alpha\geq 0$.
Zhou et al. \cite{zpla2022} characterized the graphs having the maximum value of certain bond incident
degree indices (including the second Zagreb index, general sum-connectivity index, and general zeroth-order Randi\'{c} index) in the class of all connected graphs with fixed order and number of pendent vertices.
Other related results can be found in \cite{aldi2018,agsa2022,crra2022,rabe2019}.

\textbf{If $I(x,y)$ is monotonically increasing on $x$ (or $y$), and $h(x)=I(a,x)-I(b,x)$ is monotonically decreasing on $x$ for any $a\geq b\geq 0$, then we call $I(x,y)$ has the property $P$.
If $I(x,y)$ has the property $P$ and satisfies that for any $a>b+1\geq 2$, $H(a,b)>0$, where $H(a,b)=a[I(a,a)-I(a-1,a)]-b[I(b+1,b)-I(b,b)]$.
Then we call $I(x,y)$ has the property $P^{*}$.}
For convenience, we say the VDB topological index has the property $P^{*}$ if its edge-weight function $I(x,y)$ has the property $P^{*}$.

Let $\Delta(G)$ and $\delta(G)$ be the maximum degree and minimum degree in $G$, respectively.
Denote by $n_{i}$ the number of vertices of $G$ with degree $i$, $m_{i,j}$ the number of edges of $G$ joining a vertex of degree $i$ and a vertex of degree $j$.
A $k$-vertex is a vertex with degree $k$.
A graph with maximum degree at most $4$ is called as a chemical graph.
A connected graph with $n$ vertices and $n+k-1$ edges is called a connected $k$-cyclic graph.

One important topic in chemical graph theory is determining the extremal $k$-cyclic (chemical) graphs with respect to VDB topological index.
In \cite{aaab2023,alal2022,adak2022,asem2023}, Ali et al. determined the extremal $k$-cyclic (chemical) graphs with respect to sigma index, symmetric division deg index, general Randi\'{c} index, respectively. Liu et al. \cite{luhg2023} considered the minimum Sombor index of $k$-cyclic (chemical) graphs. The first general Zagreb index and the first multiplicative Zagreb index of $k$-cyclic graphs was determined by Bianchi et al. \cite{bcpt2015}.
Other related results can be found in \cite{addi2019,ghas2020,licf2021,tome2019,tiom2022}.

Denote by $\mathcal{G}_{n,k}$ (resp. $\mathcal{CG}_{n,k}$) the set of $k$-cyclic graphs (resp. $k$-cyclic chemical graphs) with $n$ vertices.
The degree set of a graph $G$ is the class of vertex degrees of $G$.
A graph whose degree set has exactly two elements is called a bidegreed graph.
Until now, to the best of my knowledge, there is no conclusion on the general VDB topological index of $k$-cyclic (chemical) graphs.
In this paper, we give a try to unify the solution for the minimum VDB index of $k$-cyclic (chemical) graphs. For a VDB topological index $TI$ with the property $P^{*}$, we obtain the minimum $k$-cyclic (chemical) graphs for $k\geq3$, $n\geq 5(k-1)$. These VDB topological indices with property $P^{*}$ include the Sombor index, the general Sombor index for $\alpha \in [\frac{1}{2},1)$, the $p$-Sombor index for $p \in (1,2]$, the general sum-connectivity index for $\alpha \in (0,1)$ and so on.

\section{\normalsize Main results}\setcounter{equation}{0}
We firstly introduce some important lemmas.
\begin{lemma}\label{l2-01}
Suppose that $I(x,y)$ is a function with the property $P$ and $TI(G)=\sum\limits_{uv\in E(G)}I(d_{u},d_{v})$.
Let $G$ be a connected graph, $u,x,v,y$ be distinct vertices in $G$ satisfied that $ux,vy\in E(G)$, $uy,vx\notin E(G)$, $d_{u}\geq d_{v}$, $d_{y}\geq d_{x}$.
Let $G^{*}=G-\{ux,vy\}+\{uy,vx\}$.
Then $TI(G^{*})\leq TI(G)$,
with equality if and only if $d_{u}=d_{v}$ or $d_{y}=d_{x}$.
\end{lemma}
\begin{proof}
Since $I(x,y)$ be a function with the property $P$ and $TI(G)=\sum\limits_{uv\in E(G)}I(d_{u},d_{v})$, then
$TI(G)-TI(G^{*})=(I(d_{u},d_{x})-I(d_{v},d_{x}))-(I(d_{u},d_{y})$ $-I(d_{v},d_{y}))\geq 0$,
with equality if and only if $d_{u}=d_{v}$ or $d_{y}=d_{x}$.
\end{proof}

\begin{lemma}\label{l2-2}{\rm\cite{lic2023}}
Let $G$ be a connected graph with $n$ vertices and $m$ edges, $I(x,y)$ be a function with the property $P^{*}$ and $TI(G)=\sum\limits_{uv\in E(G)}I(d_{u},d_{v})$.
If $G$ is a graph achieving the least value of $TI(G)$, then $G$ is an almost regular graph, i.e.,
$\Delta(G)-\delta(G) \leq 1$.
\end{lemma}

Since $m=n+k-1$ in $\mathcal{G}_{n,k}$, then by Lemma \ref{l2-2}, we have
\begin{corollary}\label{c2-3}
Let $I(x,y)$ be a function with the property $P^{*}$ and $TI(G)=\sum\limits_{uv\in E(G)}I(d_{u},d_{v})$.
For $k\geq1$, if $G$ is a graph achieving the least value of $TI(G)$ over $\mathcal{G}_{n,k}$, then $\delta(G)\geq 2$.
\end{corollary}

\begin{lemma}\label{l2-4}{\rm\cite{addi2019}}
For $n\geq 5(k-1)$, if $G\in \mathcal{G}_{n,k}$ such that $\delta(G)\geq 2$ and $\Delta(G)\geq 4$, then
$n_{2}(G)\geq 4$.
\end{lemma}

\begin{lemma}\label{l2-5}{\rm\cite{caha2010}}
For $n\geq 5(k-1)$, if $G\in \mathcal{G}_{n,k}$ such that $\delta(G)\geq 2$ and $\Delta(G)\geq 4$, then
$m_{2,2}(G)\geq 1$.
\end{lemma}

\begin{lemma}\label{l2-6}
Let $I(x,y)$ be a function with the property $P^{*}$ and $TI(G)=\sum\limits_{uv\in E(G)}I(d_{u},d_{v})$.
For $k\geq3$ and $n\geq 5(k-1)$, if $G$ is a graph achieving the least value of $TI(G)$ over $\mathcal{G}_{n,k}$, then $\Delta(G)=3$.
\end{lemma}
\begin{proof}
Suppose to the contrary that $G$ is a counter-example, i.e., $\Delta(G)\geq 4$,
since $k\geq3$ and $G\in \mathcal{G}_{n,k}$.
By Corollary \ref{c2-3}, $\delta(G)\geq 2$.
By Lemma \ref{l2-5}, $m_{2,2}\geq 1$.
Thus $n_{2}\neq 0$ and $n_{4}\neq 0$.
Since $G\in \mathcal{G}_{n,k}$, $G$ is a connected graph with $n$ vertices and $n+k-1$ edges.
Then by Lemma \ref{l2-2}, $G$ is an almost regular graph, which is a contradiction with $n_{2}\neq 0$ and $n_{4}\neq 0$. Therefore, $\Delta(G)=3$.
\end{proof}

\begin{lemma}\label{l2-7}
Let $I(x,y)$ be a function with the property $P^{*}$ and $TI(G)=\sum\limits_{uv\in E(G)}I(d_{u},d_{v})$.
For $k\geq3$ and $n\geq 5(k-1)$, if $G$ is a graph achieving the least value of $TI(G)$ over $\mathcal{G}_{n,k}$, then
$m_{2,3}=2$.
\end{lemma}
\begin{proof}
As $G\in \mathcal{G}_{n,k}$ ($k\geq3$, $n\geq 5(k-1)$) and $I(x,y)$ is a function with the property $P^{*}$, $G$ is a bidegreed graph with degree set $\{2,3\}$ by Corollary \ref{c2-3} and Lemma \ref{l2-6}.
As $G$ is a connected bidegreed graph, $m_{2,3}>0$.
Since $\sum\limits_{1\leq j\leq 4, j\neq i}m_{i,j}+2m_{i,i}=in_{i}$ for $i=1,2,3,4$, then $2m_{2,2}+m_{2,3}=2n_{2}$.
As $k\geq3$, $m_{2,3}\geq 2$ must be an even number.

If $m_{2,3}\geq 4$, we can always find such two non-adjacent $(2,3)$-edges in $G$, by using the transformation of Lemma \ref{l2-01}, the graph $G^{*}$ is still connected simple graph and $G^{*}\in \mathcal{G}_{n,k}$ (Note that we can not use the transformation of Lemma \ref{l2-01} if $G^{*}$ is not connected simple graph). In this case $TI(G)-TI(G^{*})=2I(2,3)-I(2,2)-I(3,3)=(I(2,3)-I(2,2))-(I(3,3)-I(2,3))>0$, which is a contradiction.
Thus $m_{2,3}=2$.
\end{proof}

\begin{figure}[ht!]
  \centering
  \scalebox{.1}[.1]{\includegraphics{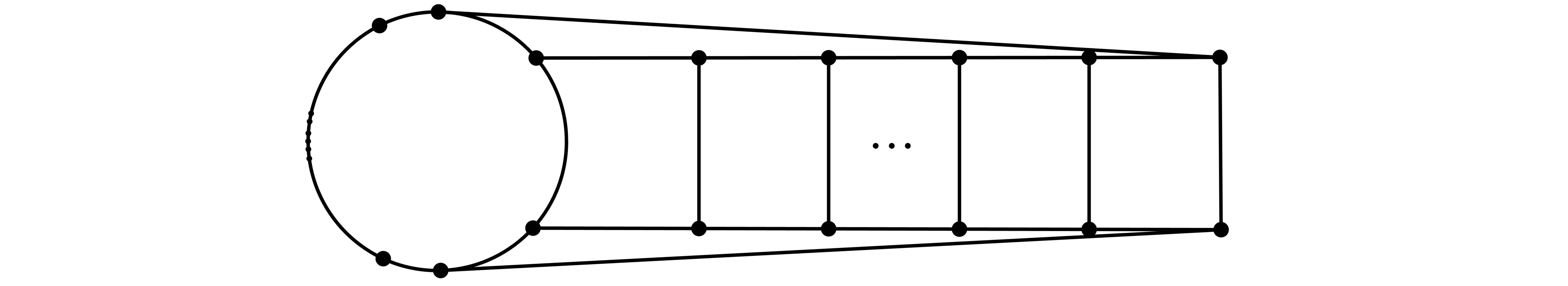}}
  \caption{An example of minimum chemical graph in Theorem \ref{t1-1}.}
 \label{fig-1}
\end{figure}

At last, we give our main results.
\begin{theorem}\label{t1-1}
Let $I(x,y)$ be a function with the property $P^{*}$ and $TI(G)=\sum\limits_{uv\in E(G)}I(d_{u},d_{v})$.
For $k\geq3$ and $n\geq 5(k-1)$, if $G$ is a graph achieving the least value of $TI(G)$ over $\mathcal{G}_{n,k}$, then
$G$ is a bidegreed graph with degree set $\{2,3\}$, and $m_{2,3}=2,m_{2,2}=n-2k+1,m_{3,3}=3k-4$. Moreover,
$TI(G)=2I(2,3)+(n-2k+1)I(2,2)+(3k-4)I(3,3)$.
\end{theorem}
\begin{proof}
As $G\in \mathcal{G}_{n,k}$ ($k\geq3$, $n\geq 5(k-1)$) and $I(x,y)$ is a function with the property $P^{*}$, $G$ is a bidegreed graph with degree set $\{2,3\}$ by Corollary \ref{c2-3} and Lemma \ref{l2-6}.
By Lemma \ref{l2-7}, $m_{2,3}=2$.

Since $\sum\limits_{i=1}^{4}n_{i}=n$ and $\sum\limits_{i=1}^{4}in_{i}=2(n+k-1)$, then by Lemma \ref{l2-6}, we have $n_{2}+n_{3}=n$ and $2n_{2}+3n_{3}=2m=2(n+k-1)$. Thus $n_{3}=2(k-1),n_{2}=n-2k+2$.
Since $\sum\limits_{1\leq j\leq 4, j\neq i}m_{i,j}+2m_{i,i}=in_{i}$ for $i=1,2,3,4$.
Then by Lemma \ref{l2-6}, we have $2m_{2,2}+m_{2,3}=2n_{2}$ and $m_{2,3}+2m_{3,3}=3n_{3}$. Since $m_{2,3}=2$, then $m_{2,2}=n-2k+1,m_{3,3}=3k-4$, and
$TI(G)=2I(2,3)+(n-2k+1)I(2,2)+(3k-4)I(3,3)$.
\end{proof}

Note that we can not use the transformation of Lemma \ref{l2-01} for the graph in Figure \ref{fig-1}
to decrease $m_{2,3}$, since the graph after the transformation is not connected graph.
Since $\mathcal{CG}_{n,k}\subseteq \mathcal{G}_{n,k}$, we have following .
\begin{corollary}\label{c1-4}
Let $I(x,y)$ be a function with the property $P^{*}$ and $TI(G)=\sum\limits_{uv\in E(G)}I(d_{u},d_{v})$.
For $k\geq3$ and $n\geq 5(k-1)$, if $G$ is a graph achieving the least value of $TI(G)$ over $\mathcal{CG}_{n,k}$, then
$G$ is a bidegreed graph with degree set $\{2,3\}$, and $m_{2,3}=2,m_{2,2}=n-2k+1,m_{3,3}=3k-4$. Moreover,
$TI(G)=2I(2,3)+(n-2k+1)I(2,2)+(3k-4)I(3,3)$.
\end{corollary}

\section{\normalsize Applications}
The Sombor index \cite{gumn2021} of graph $G$ was defined as $SO(G)=\sqrt{d_{u}^{2}+d_{v}^{2}}$.
A review about the Sombor index can be found in \cite{lgyh2021}.
Let $I(x,y)=\sqrt{x^{2}+y^{2}}$. Since $\sqrt{x^{2}+y^{2}}$ is a function with the property $P^{*}$, then by Theorem \ref{t1-1}, we have the following corollaries immediately, which solves the open problem of \cite{luhg2023}.

\begin{corollary}\label{c1-2}
For $k\geq3$ and $n\geq 5(k-1)$, if $G$ is a graph achieving the least value of the Sombor index over $\mathcal{G}_{n,k}$, then
$G$ is a bidegreed graph with degree set $\{2,3\}$, and $m_{2,3}=2,m_{2,2}=n-2k+1,m_{3,3}=3k-4$. Moreover, $SO(G)=(2n+5k-10)\sqrt{2}+2\sqrt{13}$.
\end{corollary}

The general Sombor index \cite{huzh2022} of graph $G$ was defined as $SO_{\alpha}(G)=\left(d_{u}^{2}+d_{v}^{2}\right)^{\alpha}$, where $\alpha\neq 0$.

\begin{lemma}\label{l3-1}
Let $x>0$, $y\geq z>0$, $f_{\alpha}(x,y)=\left(x^{2}+y^{2}\right)^{\alpha}$, $\phi_{\alpha}(x,y,z)=f_{\alpha}(x,y)-f_{\alpha}(x,y-z)$.

{\rm(i)} If $\alpha\in (0,+\infty)$, then $f(x,y)$ is strictly increasing with $x$ (resp. $y$); if $\alpha\in (-\infty,0)$, then $f_{\alpha}(x,y)$ is strictly decreasing with $x$ (resp. $y$).

{\rm(ii)} If $\alpha\in (0,1)$, then $\phi_{\alpha}(x,y,z)$ is strictly decreasing with $x$.
\end{lemma}
\begin{proof}
{\rm(i)}
Since $x>0, y>0$, then
\begin{align*}
 \frac{\partial f_{\alpha}(x,y)}{\partial x}=&2\alpha x \left(x^{2}+y^{2}\right)^{\alpha-1}.
\end{align*}
Thus we have $\frac{\partial f_{\alpha}(x,y)}{\partial x}>0$ if $\alpha\in (0,+\infty)$; $\frac{\partial f_{\alpha}(x,y)}{\partial x}<0$ if $\alpha\in (-\infty,0)$.

{\rm(ii)}
Since $x>0, y\geq z>0$, and $\alpha\in (0,1)$, we have
\begin{align*}
 \frac{\partial \phi_{\alpha}(x,y,z)}{\partial x}=&2\alpha x\left(\left(x^{2}+y^{2}\right)^{\alpha-1}-\left(x^{2}+(y-z)^{2}\right)^{\alpha-1}\right)<0.
\end{align*}

This completes the proof.
\end{proof}

\begin{corollary}\label{c3-2}
For $\alpha\in [\frac{1}{2},1)$, $k\geq3$ and $n\geq 5(k-1)$, if $G$ is a graph achieving the least value of general Sombor index over $\mathcal{G}_{n,k}$, then
$G$ is a bidegreed graph with degree set $\{2,3\}$, and $m_{2,3}=2,m_{2,2}=n-2k+1,m_{3,3}=3k-4$. Moreover, $SO_{\alpha}(G)=2\times 13^{\alpha}+(n-2k+1)\times 8^{\alpha}+(3k-4)\times 18^{\alpha}$.
\end{corollary}
\begin{proof}
In this case, $I(x,y)=\left(x^{2}+y^{2}\right)^{\alpha}$ for $\alpha\in [\frac{1}{2},1)$.
$\frac{\partial I(x,y)}{\partial x}=2\alpha x(x^{2}+y^{2})^{\alpha-1}$;
$\frac{\partial I(x,y)}{\partial y}=2\alpha y(x^{2}+y^{2})^{\alpha-1}$.
Next we proof that for any $a>b+1\geq 2$, $H(a,b)>0$, where $H(a,b)=a[I(a,a)-I(a-1,a)]-b[I(b+1,b)-I(b,b)]$.
By the \textbf{Lagrange mean value theorem}, there exists $\zeta\in (a-1,a)$, such that
$I(a,a)-I(a-1,a)=2\alpha \zeta(\zeta^{2}+a^{2})^{\alpha-1}$;
there exists $\xi\in (b,b+1)$, such that
$I(b+1,b)-I(b,b)=2\alpha \xi(\xi^{2}+b^{2})^{\alpha-1}$.
Thus $H(a,b)=2\alpha [a\zeta(\zeta^{2}+a^{2})^{\alpha-1}-b\xi(\xi^{2}+b^{2})^{\alpha-1}]
>2\alpha [a\zeta(a^{2}+a^{2})^{\alpha-1}-b\xi(b^{2}+b^{2})^{\alpha-1}]
=\alpha 2^{\alpha}(a^{2\alpha-1}\zeta-b^{2\alpha-1}\xi)>0$ for $\alpha\in [\frac{1}{2},1)$.

Combine with Lemma \ref{l3-1}, we know that $\left(x^{2}+y^{2}\right)^{\alpha}$
has the property $P^{*}$ for $\alpha\in [\frac{1}{2},1)$. By Theorem \ref{t1-1}, we have this conclusion.
\end{proof}

The $p$-Sombor index \cite{trdo2021} of graph $G$ was defined as $S_{p}(G)=\left(d_{u}^{p}+d_{v}^{p}\right)^{\frac{1}{p}}$, where $p\neq 0$.

\begin{lemma}\label{l3-2}
Let $p\geq 1$, $x>0$, $y\geq z>0$, $g(x,y)=\left(x^{p}+y^{p}\right)^{\frac{1}{p}}$, $\varphi(x,y,z)=g(x,y)-g(x,y-z)$ is strictly decreasing with $x$.
\end{lemma}
\begin{proof}
Since $p\geq 1$, then
\begin{align*}
 \frac{\partial \varphi(x,y,z)}{\partial x}=& x\left(x^{p}+y^{p}\right)^{\frac{1}{p}-1}-x\left(x^{p}+(y-z)^{p}\right)^{\frac{1}{p}-1}<0.
\end{align*}
Thus $\varphi(x,y,z)$ is strictly decreasing with $x$.
\end{proof}

\begin{corollary}\label{c3-3}
For $p\in (1,2]$, $k\geq3$ and $n\geq 5(k-1)$, if $G$ is a graph achieving the least value of $p$-Sombor index over $\mathcal{G}_{n,k}$, then
$G$ is a bidegreed graph with degree set $\{2,3\}$, and $m_{2,3}=2,m_{2,2}=n-2k+1,m_{3,3}=3k-4$. Moreover, $S_{p}(G)=2\times \left(2^{p}+3^{p}\right)^{\frac{1}{p}}+(n-2k+1)\times \left(2^{p}+2^{p}\right)^{\frac{1}{p}}+(3k-4)\times \left(3^{p}+3^{p}\right)^{\frac{1}{p}}$.
\end{corollary}
\begin{proof}
In this case, $I(x,y)=\left(x^{p}+y^{p}\right)^{\frac{1}{p}}$ for $p\in (1,2]$.
Thus
$\frac{\partial I(x,y)}{\partial x}=x\left(x^{p}+y^{p}\right)^{\frac{1}{p}-1}$;
$\frac{\partial I(x,y)}{\partial y}=y\left(x^{p}+y^{p}\right)^{\frac{1}{p}-1}$.
Next we proof that for any $a>b+1\geq 2$, $H(a,b)>0$, where $H(a,b)=a[I(a,a)-I(a-1,a)]-b[I(b+1,b)-I(b,b)]$.
By the \textbf{Lagrange mean value theorem}, there exists $\zeta\in (a-1,a)$, such that
$I(a,a)-I(a-1,a)=\zeta\left(\zeta^{p}+a^{p}\right)^{\frac{1}{p}-1}$;
there exists $\xi\in (b,b+1)$, such that
$I(b+1,b)-I(b,b)=\xi\left(\xi^{p}+b^{p}\right)^{\frac{1}{p}-1}$.
Thus $H(a,b)=a\zeta\left(\zeta^{p}+a^{p}\right)^{\frac{1}{p}-1}-b\xi\left(\xi^{p}+b^{p}\right)^{\frac{1}{p}-1}
>a\zeta\left(a^{p}+a^{p}\right)^{\frac{1}{p}-1}-b\xi\left(b^{p}+b^{p}\right)^{\frac{1}{p}-1}
=2^{\frac{1}{p}-1}(a^{2-p}\zeta-b^{2-p}\xi )>0$ for $p\in (1,2]$.

Combine with Lemma \ref{l3-1}, we know that $\left(x^{p}+y^{p}\right)^{\frac{1}{p}}$
has the property $P^{*}$ for $p\in (1,2]$. By Theorem \ref{t1-1}, we have this conclusion.
\end{proof}

The general sum-connectivity index \cite{zhtr2010} of graph $G$ was defined as $\chi_{\alpha}(G)=\left(d_{u}+d_{v}\right)^{\alpha}$, where $\alpha\neq 0$.
Similar to the proof of Corollary \ref{c3-2} and Corollary \ref{c3-3}, we have

\begin{corollary}\label{c3-4}
For $\alpha\in (0,1)$, $k\geq3$ and $n\geq 5(k-1)$, if $G$ is a graph achieving the least value of general sum-connectivity index over $\mathcal{G}_{n,k}$, then
$G$ is a bidegreed graph with degree set $\{2,3\}$, and $m_{2,3}=2,m_{2,2}=n-2k+1,m_{3,3}=3k-4$. Moreover, $\chi_{\alpha}(G)=2\times 5^{\alpha}+(n-2k+1)\times 4^{\alpha}+(3k-4)\times 9^{\alpha}$.
\end{corollary}

\section{\normalsize Concluding remarks}\setcounter{equation}{0}
In this paper, we try to unify the solution for the minimum VDB index of $k$-cyclic graphs. For a VDB topological index $TI$ with the property $P^{*}$, we obtain the minimum $k$-cyclic graphs for $k\geq3$, $n\geq 5(k-1)$. These VDB topological indices with property $P^{*}$ include the Sombor index, the general Sombor index for $\alpha \in [\frac{1}{2},1)$, the $p$-Sombor index for $p \in (1,2]$, the general sum-connectivity index for $\alpha \in (0,1)$ and so on.
We do not need to deal with the topological indices one by one separately.
However, we only solve the problem for the VDB indices with the nice property $P^{*}$.
We hope that in the future we can find more nice properties that can cover more topological indices.

\section*{\normalsize Acknowledgement}
We take this opportunity to thank the anonymous reviewers for their critical reading of the manuscript
and suggestions which have immensely helped us in getting the article to its present form.


\end{document}